\documentclass[leqno]{amsart}
\usepackage[a-2u]{pdfx}

\usepackage{lmodern}
\usepackage{iftex}
\ifpdftex
\usepackage[utf8]{inputenc}
\usepackage[T1]{fontenc}
\usepackage{textcomp}
\fi

\usepackage{amsfonts}       
\usepackage{bbding}         
\usepackage{bm}             
\usepackage{graphicx}       
\usepackage{fancyvrb}       
\usepackage{dcolumn}        
\usepackage{booktabs}       
\usepackage{xcolor}         
\usepackage{mathtools}
\usepackage{amssymb}
\usepackage{calc}
\usepackage{enumitem}
\usepackage{dirtytalk}
\usepackage{thmtools}
\usepackage{relsize}
\usepackage[most]{tcolorbox}

\newcommand{\rhoharm}{\rho_{\mathrm{harm}}}
\newcommand{\rhoexp}{\rho_{\mathrm{exp}}}

\tcbset{
    frame code={},
    center title,
    left=0pt,
    right=0pt,
    top=0pt,
    bottom=0pt,
    colback=red!20,
    colframe=white,
    width=\dimexpr\textwidth\relax,
    enlarge left by=0mm,
    boxsep=5pt,
    arc=0pt,outer arc=0pt,
    }

\usepackage{comment}

\hypersetup{unicode}
\hypersetup{breaklinks=true}

\theoremstyle{plain}
\newtheorem{theorem}{Theorem}
\newtheorem{lemma}[theorem]{Lemma}

\newtheorem{corollary}[theorem]{Corollary}

\theoremstyle{definition}

\theoremstyle{remark}

\newcommand{\N}{\mathbb{N}}

\newcommand{\R}{\mathbb{R}}

\newcommand{\M}{M}

\newcommand{\Mi}{\M_{\infty}}

\newcommand{\PriSp}{\Mi\times2^\N}

\newcommand{\ra}{\rightarrow}

\DeclareMathOperator{\diam}{diam}

\DeclareMathOperator{\osc}{osc}

\newcommand{\s}[1]{\{#1\}}
\newcommand{\bs}[1]{\big\{#1\big\}}
\newcommand{\Bs}[1]{\Big\{#1\Big\}}
\newcommand{\p}[1]{(#1)}
\newcommand{\bp}[1]{\big(#1\big)}
\newcommand{\Bp}[1]{\Big(#1\Big)}

\newcommand{\es}{\emptyset}

\newcommand{\cl}[1]{\lceil#1\rceil}

\newcommand{\ab}[1]{|#1|}
\newcommand{\bab}[1]{\big|#1\big|}
\newcommand{\Bab}[1]{\Big|#1\Big|}

\newcommand{\pn}[2]{\|#1\|_{#2}}

\newcommand{\Bpn}[2]{\Big\|#1\Big\|_{#2}}
\newcommand{\pnp}[1]{\pn{#1}{p}}

\newcommand{\Bpnp}[1]{\Bpn{#1}{p}}

\newcommand{\pni}[1]{\pn{#1}{\infty}}

\def \e {\boldsymbol e}
\def \f {\boldsymbol f}
\def \h {\boldsymbol h}
\def \g {\boldsymbol g}
\newcommand{\di}{\,\mathrm{d}}

\newcommand{\app}{{^\wedge}}
\newcommand{\appinf}{{^\wedge\ast}}

\newcommand{\uhr}{{\upharpoonright}}

\newcommand{\overeq}[2]{\stackrel{\mathclap{\normalfont\mbox{$\scriptscriptstyle#2$}}}{#1}}

\makeatletter
\newcommand{\littletaller}{\mathchoice{\vphantom{\big|}}{}{}{}}

\newcommand\@rest[3]{%
    #1%
    \mathclose{%
        \sbox0{$\m@th#2\left.\kern-\nulldelimiterspace\vphantom{#1}\littletaller\right.$}%
        \sbox2{$\m@th#2\upharpoonright$}%
        \dimen0=\ht0
        \advance\dimen0 by -\ht2
        \raisebox{\dimen0}{\usebox2}%
    }_{#3}%
}

\newcommand{\rest}[2]{\mathpalette\@rest@helper{{#1}{#2}}}
\def\@rest@helper#1#2{\@restaux#1#2}
\def\@restaux#1#2#3{\@rest{#2}{#1}{#3}}
\makeatother

\numberwithin{equation}{section}

\begin{document}
\title[Typical martingale diverges on a co-$\sigma$-porous set]{Typical martingale diverges on a co-$\sigma$-porous set}

\author{Antonín Hejný}

\address{Antonín Hejný \\
Charles University\\
Faculty of Mathematics and Physics\\
Prague, Czech Republic}
\email{antonin.hejny@gmail.com}


\begin{abstract}
We consider the space of $L^p$-bounded martingales on the Cantor space for $p\in[1,\infty]$. Equipping the Cantor space with two standard metrics, we show that under both, a typical martingale diverges at a typical point. By a ``typical martingale'', we mean an element of a co-$\sigma$-porous set when $p\in[1,\infty)$ and of a co-porous set when $p=\infty$. Specifically, we show that while a typical martingale diverges on a co-$\sigma$-porous set with respect to the first metric, no martingale diverges on a co-$\sigma$-porous set with respect to the second metric. Finally, we investigate the key factors determining whether the divergence set can be co-$\sigma$-porous.
\end{abstract}

\maketitle
\section{Introduction}

Doob's Martingale Convergence Theorem (see, e.g., \cite[Theorem~1.3.2.8, p.\ 25]{spaces} or \cite[275G]{FR2}) asserts that every $L^1$-bounded martingale converges almost surely with respect to the underlying probability measure. When the probability space is endowed with a natural topological structure, one may also examine the size of the convergence set from the perspective of Baire category. This question was studied in \cite{spu-zel} and \cite{kal-spu} for martingales on the Cantor space $2^\N$, where it was shown that for a comeager family of martingales in $M_p$, the convergence set is topologically meager (i.e., of first Baire category).

The present paper investigates the behavior of a typical martingale from the point of view of $\sigma$-porosity. 

The organization is as follows. In Section~\ref{sec:preliminaries}, we recall the necessary background on martingales, the Cantor space, and porosity. Section~\ref{sec:main_results} contains the statements of the main theorems, whose proofs are detailed in Section~\ref{sec:proofs}.

\section{Preliminaries}
\label{sec:preliminaries}

\subsection{Martingales}
Let $(\Omega,\Sigma,P)$ be a probability space. That is, $\Omega$ is a set, $\Sigma$ a $\sigma$-algebra of subsets of $\Omega$ and $P$ a probability measure defined on the $\sigma$-algebra $\Sigma$. A \emph{filtration} is an increasing sequence $(\Sigma_n)$ of $\sigma$-subalgebras of $\Sigma$. Denote by $\Sigma_\infty$ the $\sigma$-algebra generated by $\bigcup_{n\in\N} \Sigma_n$.

In the sequel, we will suppose that the described objects are fixed.

A \emph{martingale adapted to the filtration $(\Sigma_n)$} is a sequence $\f=(f_n)$ of functions with the following two properties.
\begin{itemize}
	\item $f_n\in L^1(\Omega,\Sigma_n,P|_{\Sigma_n})$ for each $n\in\N$.
	\item $\int_E f_n\di P=\int_E f_m\di P$ whenever $n\le m$ and $E\in \Sigma_n$.
\end{itemize}

In the sequel, we will write $L^1(\Sigma_n)$ for short instead of $L^1(\Omega,\Sigma_n,P|_{\Sigma_n})$ (and similarly for other $L^p$ spaces).
It is easy to check from the definitions that $\|f_n\|_{L^1(\Sigma_n)}\le \|f_m\|_{L^1(\Sigma_m)}$ for $n\le m$.
A martingale $(f_n)$ is called \emph{$L^1$-bounded} if $\sup_n \|f_n\|_{L^1(\Sigma_n)}<\infty$. If the equality in the second condition of the martingale definition is replaced by an inequality $\leq$, we call the random sequence a \emph{submartingale}.

\subsection{Spaces of martingales}
Let $\f=(f_n)$ be a martingale adapted to the sequence $(\Sigma_n)$ and $1\le p\le\infty$. The martingale $\f$ is called \emph{$L^p$-bounded} if $f_n\in L^p(\Sigma_n)$ for each $n\in\N$ and, moreover, $\sup\|f_n\|_{L^p}<\infty$.
The space of all $L^p$-bounded martingales will be denoted by $M_p$. If we equip $M_p$ with the norm
\begin{equation*}
\|\f\|_p\coloneqq\sup_{n\in\N}  \|f_n\|_{L^p},
\end{equation*}
it will become a Banach space.

This definition follows \cite[Definition 1.3.3 on p. 13]{spaces} with notation from \cite[Section 1]{Tr}.

Notice that $\|f_n\|_{L^p}\le\|f_m\|_{L^p}$ whenever $n\le m$. This well-known fact follows, for example, from \cite[Remark~2 on p.~11]{spaces}.

Furthermore, if $\rho$ is a metric on $\Omega$, we define $\M_p\times(\Omega,\rho)$ to be the product space with the metric $((\f,\omega),(\g,\nu))\mapsto\max\s{\pnp{\f-\g},\rho(\omega,\nu)}$.

\subsection{Cantor space}
We denote $[n]\coloneqq\{1,2,\ldots,n\}$ for $n\in \N$, $[n:m]\coloneqq\{n+1,n+2,\ldots,m\}$ for $n,m\in\N_0$ with $n<m$. 
For $\mathfrak C\in \N\cup\{\N\}$ we write $2^{\mathfrak C}\coloneqq\{0,1\}^{\mathfrak C}$ for the respective product set. Similarly, $2^{<\mathfrak C}\coloneqq\bigcup_{c<\mathfrak C}\s{0, 1}^c$ for $\mathfrak C\in\N\cup\s{\N}$ (analogously for $>$) and
$2^{\leq\mathfrak C}\coloneqq\bigcup_{c\leq\mathfrak C}\s{0, 1}^c$ for $\mathfrak C\in\N\cup\s{\N}$ (analogously for $\geq$).
Let $u_{\leq n}\coloneqq(u_k)_{k=1}^n$ for $n\in\N$ and $u\in2^{\leq\N}\cap2^{\geq n}$  (analogously for $u_{<n}$). In addition, $u\app v$ is the concatenation of $u$ and $v$, where $u\in 2^{<\N}$ and $v\in2^{\leq\N}$. 
We write
$u(\app v)^n\coloneqq u\underbrace{\app v\app v\dotsb \app v}_{n\text{ times}}$ for $u, v\in2^{<\N}$ and $n\in\N$ and
$u\appinf\coloneqq\s{v\in 2^{\geq n}\cap2^{\leq\N};v_{\leq n}=u}$ for $u\in 2^{n}$, $n\in\N$. 
Let $\nu\appinf\coloneqq\s{\nu}$ for $\nu\in2^\N$ and
$A\appinf\coloneqq\bigcup\s{u\appinf ; u\in A}$ for $A\subset 2^{\leq\N}$.

Consider $2=\{0,1\}$ with the discrete topology. Then the topological space $2^\N$ is homeomorphic to the usual Cantor set in the unit interval and can be endowed with a metric given as 
\begin{equation*}
\psi_\varphi(\eta, \nu)\coloneqq
\begin{cases}
\max\s{\varphi(n); \eta_n \neq \nu_n,\, n\in\N}, & \text{for } \eta\neq\nu,\\
0, & \text{for } \eta=\nu,
\end{cases}
\end{equation*}
where $\varphi\colon\N\ra\R^+$ is any decreasing function satisfying $\varphi(n)\ra0$. With this metric, the space $2^\N$ becomes an ultrametric compact space. The canonical basis of this compact space is given as $\s{2^\N\cap u\appinf; u\in 2^{<\N}}$. We consider two such metrics: $\rhoharm\coloneqq\psi_{n\mapsto1/n}$ and $\rhoexp\coloneqq\psi_{n\mapsto2^{-n}}$. We denote open balls with respect to those metrics as $B_{\rhoharm}(x,r)$ and $B_{\rhoexp}(x,r)$, where $x$ is the center and $r$ is the radius.

Moreover, we have a canonical probability measure $P$ on $2^\N$ given as the product measure $P = \prod_{n\in\N} P_n$, where $P_n$ is the equally distributed measure on $2$.
For any metric $\rho$ on $2^\N$ and $u, v\in 2^{\leq\N}$, we further set
\begin{align*}
    \rho(u,v)&\coloneqq\sup_{s\in u\appinf,t\in v\appinf}\rho (s,t),\\
    P(v)&\coloneqq P\bp{2^\N\cap v\appinf}.
\end{align*}

When working with the Cantor space, we will use the following simple lemma.

\begin{lemma}
\label{ultrametric_open_balls_of_same_size_lemma}
Let $(X, \psi)$ be an ultrametric space, $x, y\in X$ and $r>0$ be arbitrary. Then either 
\begin{equation*}
B(x, r)=B(y, r)\hspace{30pt}\text{or}\hspace{30pt}B(x, r)\cap B(y, r)=\es.
\end{equation*}
\end{lemma}

\subsection{Martingales on the Cantor space}
Consider now the Cantor set $2^\N$ with the probability measure $P$ and the sequence of $\sigma$-algebras $\Sigma_n=\sigma(\{u\appinf\cap2^\N; u\in 2^n\})$, $n\in\N$. Then $\Sigma_\infty$ equals the family $\mathcal{B}(2^\N)$ of all Borel sets in $2^\N$. Let $\f=(f_n)_{n\in\N}$ be a martingale adapted to the filtration $(\Sigma_n)$.
Then for each $n\in\N$ and $A=u\appinf\in\Sigma_n$, where $u\in 2^n$, we have that $f_n$ is constant on $A$ due to the $\Sigma_n$-measurability of $f_n$ (the set $A$ is an atom of $\Sigma_n$). Hence, we may write $\varphi(u)$ as the constant value of $f_n$ on $2^\N\cap u\appinf$ for any $u\in 2^{<\N}$. By the martingale property of $\f$ we obtain
\begin{equation}
    \label{eq:prumer}
\varphi(u)=\frac12\bp{\varphi(u\app 0)+\varphi(u\app 1)},\quad u\in 2^{<\N}.
\end{equation}
Indeed, $P(u)=2P(u\app 0)=2P(u\app 1)$, and thus for $u\in2^{<\N}$ we have
\[
\begin{aligned}
\varphi(u)&=\frac{1}{P(u)}\int_{2^\N\cap u\appinf} f_n\di P=
\frac{1}{P(u)}\bp{\int_{2^\N\cap u\app 0\appinf} f_{n+1}\di P+\int_{2^\N\cap u\app 1\appinf} f_{n+1}\di P}\\
&=\frac{1}{P(u)}\bp{P(u\app 0)\varphi(u\app 0)+P(u\app 1)\varphi(u\app 1)}=\frac12\bp{\varphi(u\app 0)+\varphi(u\app 1)}.
\end{aligned}
\]

Conversely, if $\varphi\colon2^{<\N}\to \R$ is a function satisfying the property~\eqref{eq:prumer}, then the sequence $\f=(f_n)_{n\in\N}$ defined as
$f_n=\varphi(u)$ on $u\appinf$ for $u\in 2^n$ is a martingale adapted to the filtration $(\Sigma_n)$ (the martingale property easily follows from \eqref{eq:prumer}).

Hence, we may identify the martingales on $(2^\N,\mathcal{B}(2^\N), P)$ adapted to $(\Sigma_n)$ with functions $\varphi\colon2^{<\N}\ra\R$ satisfying \eqref{eq:prumer}.

If $(x_n)$ is a sequence of real numbers, we set

\begin{equation*}
\osc (x_n)\coloneqq\lim_{n\to\infty} \diam(\{x_k; k\ge n\}),
\end{equation*}
where $\diam$ stands for the diameter of the respective set. Hence 
\begin{equation*}
\osc(x_n)=\limsup x_n-\liminf x_n,    
\end{equation*}
provided the right-hand side is defined.

If $\f\colon2^{<\N}\to \R$ is a martingale and $u\in 2^n$ for some $n\in\N$ is given, we define
\[
\f\uhr_{u\appinf}=\begin{cases}
                            \f(v),& v\in 2^{<\N}\cap u\appinf,\\
                            0,& v\in 2^{<\N}\cap (2^n\setminus\{u\})\appinf.\\
                            \end{cases}
\]
The definition of $\f\uhr_{u\appinf}$ on $2^{<n}$ is then given by the formula \eqref{eq:prumer} and the already defined values on $2^n$.

In the sequel, we will need the following easy results on the oscillation of martingales.

\begin{lemma}
\label{near_martingales_lemma}
Let $\f, \g$ be martingales in $\M_\infty$ and $\nu\in 2^\N$. Let also $r>2d>0$ be such that $\|\f-\g\|_\infty<d$ and $\osc f_n(\nu)\geq r$. Then $\osc g_n(\nu)> r-2d$.
\end{lemma}

\begin{lemma}
\label{martingale_sum_oscillation_lemma}
Let $\f,\g\in\M_\infty$ and $\nu\in 2^\N$. Let also $r>d>0$ be such that 
\begin{equation*}
\osc (f_n(\nu))\leq d\hspace{30pt}\text{and}\hspace{30pt}\osc(g_n(\nu))\geq r.
\end{equation*}
Then it holds that
\begin{equation*}
\osc(f_n+g_n)(\nu)\geq r-d.
\end{equation*}
\end{lemma}

\subsection{Porous and $\sigma$-porous set}
Let $(X, \psi)$ be a metric space and let $x\in X$, $A\subset X$ and $R>0$ be arbitrary. Then we define 
\begin{equation*}
\gamma(A, R, x)=\sup\bp{\s{r>0;\exists y\in X:B(y, r)\subset B(x, R)\setminus A}\cup\s{0}}.
\end{equation*}
Further, we define porosity (a.k.a. lower porosity) of $A$ at $x$ as
\begin{equation*}
p(A, x)=2\cdot\liminf_{R\ra0^+}\frac{\gamma(A, R, x)}{R}.
\end{equation*}

We say that $A$ is porous ($c$-porous) at $x$ if $p(A, x)>0$ ($p(A, x)\geq c$). For $B\subset X$, we say $A$ is porous ($c$-porous) on $B$ if it is porous ($c$-porous) at each point of $B$. We say $A$ is porous ($c$-porous) if it is porous on $X$. We say $A$ is $\sigma$-porous (at $B$) if it is a countable union of porous sets (at $B$). The following two lemmata are simple results that we will use in the sequel.

\begin{lemma}
\label{porosity_condition_lemma}
Let $(X,\psi)$ be a metric space and let $x\in X$, $A\subset X$ and $\alpha>0$. Then $A$ is $2\alpha$-porous at $x$ if and only if
\begin{equation*}
\forall\varepsilon\in(0,\alpha)\exists r_0>0\forall r\in(0,r_0)\exists y\in X:B\bp{y,(\alpha-\varepsilon)\cdot r}\subset B(x,r)\setminus A.
\end{equation*}
\end{lemma}

\begin{lemma}
\label{constant_sigma_porosity}
Let $(X,\psi_1)$ and $(X,\psi_2)$ be metric spaces and let $c,d>0$ be such that $c\cdot\psi_1(x,y)\leq\psi_2(x,y)\leq d\cdot\psi_1(x,y)$ for all $x,y\in X$. Then porous sets in $(X,\psi_1)$ and $(X,\psi_2)$ coincide.
\end{lemma}

In further text, we will also be using the following theorem, which can be found in e.g. \cite[Proposition~2.2 on p.~512]{Zajicek}.

\begin{theorem}
\label{zajicek_theorem}
Let $(X,\psi)$ be a metric space and let $A\subset X$. Then $A$ is $\sigma$-porous if and only if $A=\bigcup_{n\in\N}F_n$, where each $F= F_n$ satisfies:
\begin{equation*}
\exists\alpha>0\exists r_0>0\forall x\in X\forall r\in(0,r_0)\exists y\in X:B(y,\alpha r)\subset B(x,r)\setminus F.
\end{equation*}
\end{theorem}

Using this theorem, we can readily obtain the following corollary.

\begin{corollary}
\label{zajicek_corollary}
Let $(X,\psi)$ be a metric space and let $A\subset X$. Let also $s_n>0$, $\beta_n\in(0,1)$, $n\in\N$ be such that $\beta_n\ra0$ and $s_n\ra0$. Then $A$ is $\sigma$-porous if and only if $A=\bigcup_{n\in\N}A_n$, where each $A_n$ satisfies:
\begin{equation*}
\forall x\in X\forall r\in(0,s_n)\exists y\in X:\overline{B\p{y,\beta_nr}}\subset B(x,r)\setminus A_n.
\end{equation*}
\end{corollary}

\begin{lemma}
\label{porous_zero_measure_cantor_lemma}
Let $\alpha>0$ and $r_0>0$ be arbitrary. Let also $F\subset 2^\N$ be such that
\begin{equation}
\label{porous_zero_measure_cantor_lemma:assumption}
\forall x\in 2^\N\forall r\in(0,r_0)\exists y\in 2^\N:B_{\rhoexp}(y,\alpha r)\subset B_{\rhoexp}(x,r)\setminus F.
\end{equation}
Then $F$ is measurable and $P(F)=0$.
\end{lemma}

\begin{proof}
Let us choose any $k, j\in\N$ such that 
\begin{equation}
\label{porous_zero_measure_cantor_lemma:kj_def}
2^{-k}\leq\alpha\hspace{40pt}\text{and}\hspace{40pt}2^{-kj}\leq r_0. 
\end{equation}
We construct $(\mathcal G_n)_{n\geq j}$ such that
\begin{flalign*}
&\forall n\geq j\ \forall G\in \mathcal G_n\ \exists y\in 2^\N:G=B_{\rhoexp}(y, 2^{-kn}),
\tag{I}\label{porous_zero_measure_cantor_lemma:I}&\\
&\forall n\geq j\ \forall G\in \mathcal G_n:P(G)=2^{-kn},
\tag{II}\label{porous_zero_measure_cantor_lemma:II}&\\
&\forall n\geq j: |\mathcal G_n|\leq(2^k-1)^{n-j}\cdot 2^{kj},
\tag{III}\label{porous_zero_measure_cantor_lemma:III}&\\
&F\subset \bigcap_{n\geq j} \bigcup \mathcal G_n,
\tag{IV}\label{porous_zero_measure_cantor_lemma:IV}
\end{flalign*}
which will conclude the proof, since $P$ is complete and it holds that
\begin{align*}
P(F)&\overeq{\leq}{(\text{\ref{porous_zero_measure_cantor_lemma:IV}})} P\bp{\bigcap_{n\geq j} \bigcup \mathcal G_n}\leq \limsup_{n\ra\infty} P\bp{\bigcup \mathcal G_n}
\overeq{\overeq{\leq}{(\text{\ref{porous_zero_measure_cantor_lemma:III}})}}{(\text{\ref{porous_zero_measure_cantor_lemma:II}}),} \limsup_{n\ra\infty} \frac{(2^k-1)^{n-j}\cdot 2^{kj}}{2^{kn}} \\
&=\limsup_{n\ra\infty}\Bp{1-\frac{1}{2^k}}^{n-j}
=0.
\end{align*}
So, let us proceed with the construction. We set 
\begin{equation}
\label{porous_zero_measure_cantor_lemma:G_def}
\mathcal G_{j}\coloneqq\bs{B_{\rhoexp}(x, 2^{-kj});x\in2^\N}.
\end{equation}
By Lemma \ref{ultrametric_open_balls_of_same_size_lemma}, $\mathcal G_{j}$ is a system of pairwise disjoint sets. Let $n>j$ arbitrary, assuming that $\mathcal G_{n-1}$ is already constructed. Then for each $G\in\mathcal G_{n-1}$, we may use \eqref{porous_zero_measure_cantor_lemma:I} to express it as $B_{\rhoexp}(x,2^{-k(n-1)})$ for some $x\in2^\N$. So, using \eqref{porous_zero_measure_cantor_lemma:assumption}, we may choose $y_G\in 2^\N$ such that 
\begin{equation*}
B_{\rhoexp}(y_G,2^{-kn})
\overeq{\subset}{\eqref{porous_zero_measure_cantor_lemma:assumption}} B_{\rhoexp}\bp{x,\frac{1}{\alpha}\cdot2^{-kn}}\setminus F
\ \overeq{\subset}{\eqref{porous_zero_measure_cantor_lemma:kj_def}} B_{\rhoexp}\bp{x, 2^{-k(n-1)}}\setminus F
=G\setminus F.
\end{equation*}
For such $G$, let us further define 
\begin{equation*}
\mathcal H_G\coloneqq\bs{B_{\rhoexp}(z, 2^{-kn});z\in G\setminus B_{\rhoexp}(y_G, 2^{-kn})}.
\end{equation*}
Then we finish the construction by setting
\begin{equation*}
    \mathcal G_n=\bigcup_{G\in\mathcal G_{n-1}}\mathcal H_G.
\end{equation*}

Let us further observe that for any $n\geq j$ and $G\in\mathcal G_n$, Lemma \ref{ultrametric_open_balls_of_same_size_lemma} guarantees us that
\begin{equation}
\label{porous_zero_measure_cantor_lemma:union}
\bigsqcup\mathcal H_G=G\setminus B_{\rhoexp}(y_G,2^{-kn}).
\end{equation}
We now finish the proof by showing that the individual properties \eqref{porous_zero_measure_cantor_lemma:I}-\eqref{porous_zero_measure_cantor_lemma:IV} hold.
\begin{itemize}
\item[``\eqref{porous_zero_measure_cantor_lemma:I}'']
This property follows from the definition of $\mathcal H_G$ and from \eqref{porous_zero_measure_cantor_lemma:G_def}.

\item[``\eqref{porous_zero_measure_cantor_lemma:II}'']
This property follows from \eqref{porous_zero_measure_cantor_lemma:I} and the definition of the Cantor space.

\item[``\eqref{porous_zero_measure_cantor_lemma:III}'']
In order to prove that this property holds, let us use \eqref{porous_zero_measure_cantor_lemma:G_def} and Lemma \ref{ultrametric_open_balls_of_same_size_lemma} to deduce that
\begin{equation*}
\ab{\mathcal G_j}=\bab{\bs{2^\N\cap u\appinf;u\in2^{kj}}}=2^{kj}.
\end{equation*}
From (\ref{porous_zero_measure_cantor_lemma:union}), we get that each $\mathcal H_G$ has $2^k-1$ elements. Hence, it holds that
\begin{equation*}
\ab{\mathcal G_n}=(2^k-1)\cdot\ab{\mathcal G_{n-1}},\hspace{15pt}n>j,
\end{equation*}
giving us \eqref{porous_zero_measure_cantor_lemma:III} by induction.

\item[``\eqref{porous_zero_measure_cantor_lemma:IV}'']
This property holds since we began with the entire $2^\N$ and in each step, $\bigcup \mathcal G_n$ is reduced only by some subset of $2^\N\setminus F$.
\end{itemize}
\end{proof}

\begin{lemma}
\label{sigma_porous_zero_measure_cantor_lemma}
Every $\sigma$-porous set in the Cantor space has zero measure.
\end{lemma}

\begin{proof}
Let $F\subset2^\N$ be $\sigma$-porous. Let $F_n,n\in\N$ be the sets guaranteed by Theorem \ref{zajicek_theorem}. Then from Lemma \ref{porous_zero_measure_cantor_lemma} and the $\sigma$-additivity of $P$, we get that
\begin{equation*}
    P(F)\leq\sum_{n\in\N}P(F_n)=\sum_{n\in\N}0=0.
\end{equation*}
\end{proof}

\section{Divergence of a typical martingale}
\label{sec:main_results}

In this section, we state all the important results. We prove them in the following section. These results are an attempt to strengthen the results of \cite{kal-spu}. Although a slightly more general context was used in \cite{kal-spu}, we interpret their results in terms of $2^\N$. Since it turns out that different topologically equivalent metrics on the Cantor space yield different results, we will consider each of the metrics separately.

\subsection{Metric derived from $\frac{1}{n}$}
We start with the $p=\infty$ case. In this scenario, \cite[Theorem~4.6 on p.~7]{kal-spu} states that the set
\begin{equation*}
    \bs{\f\in \M_\infty; \s{\nu\in2^\N;\osc f_n(\nu)>0} \text{ is comeager}}
\end{equation*}
is comeager in $\M_\infty$. In the following theorem, we show that even more martingales diverge on an even greater set.

\begin{restatable}{theorem}{MinftyIsCoPorous}
\label{m_infty_is_co_porous}
The set
\begin{equation*}
R\coloneqq\bs{\f\in \M_\infty; \s{\nu\in2^\N;\osc f_n(\nu)>0} \text{ is co-$\sigma$-porous in $(2^\N,\rhoharm)$}}
\end{equation*}
is co-1-porous in $\M_\infty$ (i.e., it is the complement of a $1$-porous set).
\end{restatable}

In any reasonable setting, the above theorem cannot be generalized in any way by replacing \emph{co-$\sigma$-porous} by \emph{co-porous} (not even by \emph{co-nowhere dense}) in the definition of $R$, as follows from the following theorem.

\begin{theorem}
\label{martingales_converge_on_dense_set_theorem}
Let $(S, \zeta)$ be a metric space, $\mathcal B(S)$ its Borel $\sigma$-algebra, and $\mu$ a probability measure on $\mathcal B(S)$. Assume that every non-empty open set in $(S, \zeta)$ has non-zero measure. Then for every non-empty open set $U\subset S$, any martingale $\f$ on $(S, \mathcal B(S),\mu)$ that is $L^p$-bounded for some $p\in[1,\infty]$ converges at some element of $U$ (i.e., the set of convergence points of $\f$ is dense in $S$).
\end{theorem}

\begin{proof}
From the inequality $\pn{\cdot}{1}\leq\pnp{\cdot}$, we get that $\f$ is $L^1$-bounded. Since any non-empty open $U\subset S$ has non-zero measure, we may use Doob's Martingale Convergence Theorem to conclude that $\f$ converges almost surely, and hence converges at least at one element of $U$.
\end{proof}

Naturally, we would expect a similar result as Theorem \ref{m_infty_is_co_porous} to hold even for $p\in[1,\infty)$. Indeed, in the second part of \cite[Theorem 4.1 on p. 5]{kal-spu}, it is stated that for all $p\in[1,\infty)$,
\begin{equation*}
\bs{\f\in\M_p;\s{\eta\in2^\N;\limsup f_n(\eta)=\infty\wedge\liminf f_n(\eta)=-\infty}\text{ is comeager}}
\end{equation*}
is comeager in $\M_p$. In this direction, we were able to make only partial progress. We show that if we keep the requirement that the set of divergence is comeager, the resulting set is co-$\sigma$-porous.

\begin{restatable}{theorem}{MpIsCoSigmaPorous}
\label{m_p_is_co_sigma_porous}
Let $p\in[1, \infty)$ be arbitrary. Then
\begin{equation*}
\bs{\f\in\M_p;\s{\eta\in2^\N;\osc f_n(\eta)=\infty}\text{ is comeager in }(2^\N,\rhoharm)}
\end{equation*}
is co-$\sigma$-porous in $\M_p$.
\end{restatable}

\subsection{Metric derived from $2^{-n}$}

First, we show that the results of \cite{kal-spu} cannot be strengthened in terms of the size of the set where individual martingales diverge.

\begin{theorem}
\label{no_divergent_martingale}
Let $p\in[1,\infty]$ and let $\f\in\M_p$. Then 
\begin{equation*}
    Q\coloneqq\bs{\nu\in2^\N;\osc f_n(\nu)>0}
\end{equation*}
is not co-$\sigma$-porous in $(2^\N,\rhoexp)$.
\end{theorem}

\begin{proof}
From the $\pn{\cdot}{1}\leq\pnp{\cdot}$ inequality, we get $\f\in\M_1$. So, by the Doob theorem, $P(Q)=0$. We can therefore use Lemma \ref{sigma_porous_zero_measure_cantor_lemma} to conclude that it cannot be co-$\sigma$-porous.
\end{proof}

Now, let us consider the $p\in[1,\infty)$ case. Since any co-porous set is co-nowhere-dense, any co-$\sigma$-porous set is comeager and both metrics $\rhoexp$, $\rhoharm$ generate the same topology, Theorems \ref{m_infty_is_co_porous} and \ref{m_p_is_co_sigma_porous} yield the following two corollaries.

\begin{corollary}
\label{m_infty_is_co_porous_corollary}
The set
\begin{equation*}
R\coloneqq\bs{\f\in \M_\infty; \s{\nu\in2^\N;\osc f_n(\nu)>0} \text{ is comeager in $(2^\N,\rhoexp)$}}
\end{equation*}
is co-1-porous in $\M_\infty$.
\end{corollary}

\begin{corollary}
\label{m_p_is_co_sigma_porous_corollary}
Let $p\in[1, \infty)$ be arbitrary. Then
\begin{equation*}
\Bs{\f\in\M_p;\bs{\eta\in2^\N;\osc f_n(\eta)=\infty}\text{ is comeager in $(2^\N,\rhoexp)$}}
\end{equation*}
is co-$\sigma$-porous in $\M_p$.
\end{corollary}

We also show a negative result in the product space.

\begin{restatable}{theorem}{MiPrNotSigmaPorous}
\label{mi_pr_not_sigma_porous}
\begin{equation*}
A\coloneqq\bs{(\f,\nu)\in\PriSp;\osc f_n(\nu)=0}
\end{equation*}
is not $\sigma$-porous in $\Mi\times(2^\N,\rhoexp)$.
\end{restatable}

\subsection{Other metrics}

We could naturally ask what the key difference between $\frac{1}{n}$ and $2^{-n}$ is that determines which sets are porous and which are not. And what about other metrics on $2^\N$? Could the results be even more surprising? 

It turns out that a very important factor is the limit behavior of the fraction between two consecutive values of the distance weighting function $\varphi$. The closer this fraction is to $1$, the more sets are porous. So, the conclusions of Theorems \ref{m_infty_is_co_porous} and \ref{m_p_is_co_sigma_porous} would hold for the metric derived from $1/\sqrt{n}$, but no martingale diverges on a co-$\sigma$-porous set if the metric is derived from $2^{-n^2}$.

The rest of this section contains formal statements of such results.

\begin{restatable}{lemma}{PowerDoesNotChangePorosity}
\label{power_does_not_change_porosity}
Let $p\in(0,\infty)$ and $(X,\xi)$ be a metric space such that $(X,\xi^p)$ is also a metric space. Let also $x\in X$ and $a>0$. Then any $A\subset X$ that is $2a$-porous in $(X,\xi^p)$ at $x$ is $2\sqrt[p]{a}$-porous in $(X,\xi)$ at $x$. Here, $\xi^p(x,y)$ simply means the value of the metric $\xi(x,y)$ raised to the power of $p$.
\end{restatable}

\begin{restatable}{theorem}{LimitPorosityCriterionTheorem}
\label{limit_porosity_criterion_theorem}
Let $\varphi_1,\varphi_2\colon\N\ra\R^+$ be decreasing functions such that $\varphi_1(n)\ra0$,  $\varphi_2(n)\ra0$ and
\begin{equation}
\label{limit_porosity_criterion_theorem:varphi12relation}
\exists n_0\in\N\forall n\geq n_0:\frac{\varphi_1(n+1)}{\varphi_1(n)}\leq\frac{\varphi_2(n+1)}{\varphi_2(n)}.
\end{equation}
Let us further denote $\psi_i\coloneqq\psi_{\varphi_i}$, $i\in\s{1,2}$. Then any set that is porous in $(2^\N,\psi_1)$ is also porous in $(2^\N,\psi_2)$.
\end{restatable}

\begin{corollary}
\label{LimitPorosityCriterionRationalCorollary}
Let $R_1,R_2\colon\N\ra\R^+$ be rational functions that are decreasing on $\N$ and satisfy that $R_1(n)\ra0$, $R_2(n)\ra0$. Then if we define $\psi_i\coloneqq\psi_{R_i}$, $i\in\s{1,2}$,
we obtain two metrics on $2^\N$ that generate the same porous sets.
\end{corollary}

\begin{proof}
Let us find $p,q\in(0,\infty)$ such that both pairs $(R_1^p,R_2)$ and $(R_2^q,R_1)$ satisfy the condition in Theorem \ref{limit_porosity_criterion_theorem}. From Lemma \ref{power_does_not_change_porosity}, we see that the porous sets in $\psi_1$ and $\psi_1^p$, as well as in $\psi_2$ and $\psi_2^q$, coincide. Theorem \ref{limit_porosity_criterion_theorem} therefore gives us that porous sets coincide in $\psi_1$ and $\psi_2$ as well.
\end{proof}

\begin{restatable}{corollary}{CantorSetDivergence}
\label{cantor_set_divergence}
Let $\xi\colon2^\N\times2^\N\ra\R^+_0$ be the Euclidean metric on the Cantor set represented as a subset of the interval $[0,1]$. That is, let
\begin{equation*}
\xi(\nu,\eta)\coloneqq\Bab{\sum_{n\in\N}2\cdot\frac{\nu_n-\eta_n}{3^n}}\hspace{20pt}\text{for any }\nu,\eta\in2^\N.
\end{equation*}
Let also $\f\in\M_1$ be arbitrary. Then in $(2^\N,\xi)$, $\f$ does not diverge on a co-$\sigma$-porous set.
\end{restatable}

\begin{proof}
As in the proof of Theorem \ref{limit_porosity_criterion_theorem}, we can prove that the co-$\sigma$-porous sets in $(2^\N,\rhoexp)$ and $(2^\N, \psi_{n\mapsto 3^{-n}})$ coincide. Also, it can be easily shown that there is a constant $c>0$ such that $\frac{1}{c}\psi_{n\mapsto 3^{-n}}\leq\xi\leq c\psi_{n\mapsto 3^{-n}}$. Thus, the statement follows from Lemma \ref{constant_sigma_porosity} and Theorem \ref{no_divergent_martingale}.
\end{proof}

\section{Proofs of the main results}
\label{sec:proofs}

This section contains the proofs of all the results from the previous section. It also contains several lemmata along with their proofs.

\begin{lemma}
\label{MinftyIsCoPorous_lemma}
Let $\f\in\M_\infty$ and $r>0$ be arbitrary. Then there exists $\g\in\Mi, \pni{\g-\f}\leq\frac{r}{2}$ such that
\begin{equation}
\label{MinftyIsCoPorous_lemma:g_condition}
\forall\varepsilon>0\exists m_0\in\N\forall m\geq m_0\forall u\in2^m\exists w\in2^{12m}\cap u\appinf:
\ab{g_{2m}(w)-g_{12m}(w)}
\geq r-\varepsilon
.
\end{equation}
\end{lemma}

\begin{proof}
We construct the desired $\g$ by induction. First, we set $\g(\es)\coloneqq \f(\es)$. Then, we consider any $n\in\N$ such that $g_{\leq n}$ is already constructed, and for any $v\in2^n$, we define
\begin{equation}
\label{MinftyIsCoPorous_lemma:d_def}
d^v\coloneqq\frac{1}{2}\cdot\min\Bs{\bab{\g(v)-\f(v)-\frac{r}{2}},\bab{\g(v)-\f(v)+\frac{r}{2}}}
\end{equation}
and let $a\in\s{0,1}$ be any such that $\f(v\app a)\geq \f(v)$ and $b\coloneqq 1-a$. Then we further define
\begin{align}
\begin{split}
\label{MinftyIsCoPorous_lemma:g}
\g(v\app a)&\coloneqq \f(v\app a)+\g(v)-\f(v)+d^v,\\
\g(v\app b)&\coloneqq \f(v\app b)+\g(v)-\f(v)-d^v.
\end{split}
\end{align}
Thus, the construction is finished. By writing
\begin{align*}
\frac{\g(v\app a)+\g(v\app b)}{2}
&\overeq{=}{\eqref{MinftyIsCoPorous_lemma:g}}\frac{\f(v\app a)+\g(v)-\f(v)+d^v+\f(v\app b)+\g(v)-\f(v)-d^v}{2}\\
&=\g(v)-\f(v)+\frac{\f(v\app a)+\f(v\app b)}{2}=\g(v)-\f(v)+\f(v)=\g(v),
\end{align*}
we get that the constructed function is a
martingale. Furthermore, if $\ab{\f(v)-\g(v)}\leq\frac{r}{2}$ for some $v\in2^{<\N}$, then
\begin{align*}
\bab{\f(v\app a)&-\g(v\app a)}\hspace{2pt}
\overeq{=}{\eqref{MinftyIsCoPorous_lemma:g}}\hspace{2pt}\bab{\f(v\app a)-\bp{\f(v\app a)+\g(v)-\f(v)+d^v}}\\
&=\bab{\f(v)-\g(v)-d^v}\hspace{2pt}
\overeq{\leq}{\eqref{MinftyIsCoPorous_lemma:d_def}}\hspace{2pt}\frac{r}{2}.
\end{align*}
Similarly, we could prove that $\ab{\f(v\app b)-\g(v\app b)}\leq\frac{r}{2}$. Since also $\ab{\f(\es)-\g(\es)}=0$, we may use the principle of mathematical induction to deduce that $\pni{\f-\g}\leq\frac{r}{2}$. 

Now we prove that \eqref{MinftyIsCoPorous_lemma:g_condition} holds. For this, we let $\varepsilon>0$ be arbitrary and set $m_0\coloneqq\cl{\log_{\frac{1}{2}}(\frac{\varepsilon}{r})}$. Let $m\geq m_0$ and $u\in2^m$ be given.
From \eqref{MinftyIsCoPorous_lemma:d_def} and \eqref{MinftyIsCoPorous_lemma:g}, it holds that
\begin{equation*}
\forall v\in2^{<\N}\forall c\in2\exists t\in2^2:\g(v\app c\app t)-\f(v\app c\app t)\leq \g(v)-\f(v),\ \f(v\app c\app t)\leq \f(v\app c).
\end{equation*}
By extending the argument, we may show that given any $k\in\N$, $v\in2^{<\N}$ and $c\in2^k$, there exists $t\in2^{2k}$ such that
\begin{equation*}
\g(v\app c\app t)-\f(v\app c\app t)\leq \g(v)-\f(v)\hspace{20pt}\text{and}\hspace{20pt} \f(v\app c\app t)\leq \f(v\app c).
\end{equation*}
By substituting $v\coloneqq\es$ and using the fact that $\f(\es)=\g(\es)$, we obtain that
\begin{equation}
\label{MinftyIsCoPorous_lemma:two_steps_small}
\forall k\in\N\forall c\in2^k\exists t_1\in2^{2k}: \g(c\app t_1)\leq \f(c\app t_1)\leq \f(c).
\end{equation}
Note that similarly, we could also prove that
\begin{equation*}
\forall k\in\N\forall c\in2^k\exists t_2\in2^{2k}:\g(c\app t_2)\geq \f(c\app t_2)\geq \f(c).
\end{equation*}
Let us now use \eqref{MinftyIsCoPorous_lemma:d_def} and \eqref{MinftyIsCoPorous_lemma:g} to infer that given any $k\in\N$ and $v\in2^k$ such that $\g(v)\geq \f(v)$, there exists $s_1\in2^k$ such that
\begin{equation}
\label{MinftyIsCoPorous_lemma:exists_small}
\g(v\app s_1)\geq \f(v\app s_1)+\frac{r}{2}\cdot(1-2^{-k})\hspace{20pt}\text{and}\hspace{20pt} \f(v\app s_1)\geq \f(v).
\end{equation}
Similarly, given any $v\in2^k$ such that $\g(v)\leq \f(v)$, there exists $s_2\in2^k$ such that
\begin{equation}
\label{MinftyIsCoPorous_lemma:exists_big}
\g(v\app s_2)\leq \f(v\app s_2)-\frac{r}{2}\cdot(1-2^{-k})\hspace{20pt}\text{and}\hspace{20pt} \f(v\app s_2)\leq \f(v).
\end{equation}
Let us now WLOG assume that $\g(u)\geq \f(u)$, the case $\g(u)\leq\f(u)$ is symmetric. Then we may use \eqref{MinftyIsCoPorous_lemma:exists_small} for $v\coloneqq u$ and $k\coloneqq m$ to find $s_1\in2^m$ such that
\begin{equation}
\label{MinftyIsCoPorous_lemma:g_uts}
\g(u\app s_1)\geq \f(u\app s_1)+\frac{r}{2}\cdot(1-2^{-m}).
\end{equation}
Then, by applying \eqref{MinftyIsCoPorous_lemma:two_steps_small} to $c\coloneqq u\app s_1$, we find $t_1\in2^{4m}$ such that
\begin{equation}
\label{MinftyIsCoPorous_lemma:g_utst}
\f(u\app s_1)\geq\f(u\app s_1\app t_1)\geq\g(u\app s_1\app t_1).
\end{equation}
Finally, we use \eqref{MinftyIsCoPorous_lemma:exists_big} for $v\coloneqq u\app s_1\app t_1$ and $k\coloneqq 6m$ to find $s_2\in2^{6m}$ such that
\begin{equation}
\label{MinftyIsCoPorous_lemma:g_utsts}
\g(u\app s_1\app t_1\app s_2)\leq \f(u\app s_1\app t_1)-\frac{r}{2}\cdot(1-2^{-6m})\leq\f(u\app s_1\app t_1)-\frac{r}{2}\cdot(1-2^{-m}).
\end{equation}
By combining \eqref{MinftyIsCoPorous_lemma:g_utst} and \eqref{MinftyIsCoPorous_lemma:g_utsts}, we get that
\begin{equation*}
\g(u\app s_1\app t_1\app s_2)\leq \f(u\app s_1)-\frac{r}{2}\cdot(1-2^{-m}).
\end{equation*}
This, together with \eqref{MinftyIsCoPorous_lemma:g_uts}, gives us that
\begin{equation*}
\ab{g_{2m}(w)-g_{12m}(w)}=\bab{\g(u\app s_1\app t_1\app s_2)-\g(u\app s_1)}\geq r\cdot(1-2^{-m}),
\end{equation*}
where $w\coloneqq u\app s_1\app t_1\app s_2\in2^{12m}$. Thus, we conclude the proof by observing that
\begin{equation*}
r\cdot(1-2^{-m})
\geq r\cdot(1-2^{-m_0})
=r-r\cdot2^{-\cl{\log_{\frac{1}{2}}(\frac{\varepsilon}{r})}}
\geq r-r\cdot2^{-\log_{\frac{1}{2}}(\frac{\varepsilon}{r})}
=r-r\cdot\frac{\varepsilon}{r}
=r-\varepsilon.
\end{equation*}
\end{proof}

\MinftyIsCoPorous*

\begin{proof}
Let $\f\in \M_\infty$ and $r>0$ be arbitrary. We take $\g\in\overline{B(\f,\frac{r}{2})}$ to be that which is guaranteed by Lemma \ref{MinftyIsCoPorous_lemma}. Then we have that given any $k\in\N$, there is some $m_k\in\N$ such that
\begin{equation*}
\forall m\geq m_k\forall u\in2^m\exists w\in2^{12m}\cap u\appinf:\ab{g_{2m}(w)-g_{12m}(w)}\geq r-\frac{1}{k}.
\end{equation*}
Given any such $u$, we denote the corresponding $w$ as $w^u$. For any $k\in\N$, we set
\begin{flalign*}
\mathcal B_k&\coloneqq\bigcup\bs{2^\N\cap(w^u)\appinf;m\geq \max\s{k,m_k}, u\in2^m},&\\
\mathcal B&\coloneqq\bigcap_{k\in\N}\mathcal B_k.&
\end{flalign*}
From the construction, each $\mathcal B_k$ is co-porous in $(2^\N,\rhoharm)$; therefore, $\mathcal B$ is co-$\sigma$-porous. We now show that for any $\eta\in\mathcal B$, $\osc g_n(\eta)\geq r$. In order to do that, let $\varepsilon>0$ be arbitrary. We fix $k_0\coloneqq\cl{1/\varepsilon}$. Then for any $k_1\geq k_0$, we use the fact that $\eta\in\mathcal B_{k_1}$ to find $m\geq \max\s{k_1,m_{k_1}}$ such that $\eta_{\leq 12m}=w^{\eta_{\leq m}}$. Thus,
\begin{equation*}
\ab{g_{2m}(\eta)-g_{12m}(\eta)}=\ab{g_{2m}(w^{\eta_{\leq m}})-g_{12m}(w^{\eta_{\leq m}})}\geq r-\frac{1}{k_1}\geq r-\frac{1}{k_0}=r-\frac{1}{\cl{1/\varepsilon}}\geq r-\varepsilon,
\end{equation*}
which implies $\osc g_n(\eta)\geq r$ since $m\geq k_1\geq k_0$ can be chosen arbitrarily large. 

So the set where the oscillation of $\g$ is at least $r$ contains $\mathcal B$ and is therefore co-$\sigma$-porous. For any $\h\in B(\g,\frac{r}{2})$, let $d\coloneqq\pni{\h-\g}<\frac{r}{2}$. By Lemma \ref{near_martingales_lemma}, for every $\eta\in\mathcal B$ we have $\osc h_n(\eta)> r-2d>0$, which shows that $\h\in R$. Hence,
\begin{equation*}
B\bp{\g,\frac{r}{2}}\subset B(\f,r)\cap R=B(\f,r)\setminus(\M_\infty\setminus R).
\end{equation*}
By Lemma \ref{porosity_condition_lemma}, $\M_\infty\setminus R$ is $1$-porous at $\f$. Since $\f\in\M_\infty$ was arbitrary, $R$ is co-$1$-porous in $\M_\infty$.
\end{proof}

\begin{lemma}
\label{martingales_are_not_sigma_porous_infty_lp_lemma}
Let $p\in[1,\infty)$, $n\in\N$ and $\f\in\M_p$ be arbitrary. Then
\begin{equation*}
    2^{n(p-1)/p}\cdot\pnp{\f}\geq\sum_{u\in2^n}\pnp{\f\uhr_{u\appinf}}.
\end{equation*}
\end{lemma}

\begin{proof}
We prove the lemma using Jensen's inequality:
\begin{align*}
2^{n(p-1)/p}\cdot\pnp{\f}
&=2^{n(p-1)/p}\sqrt[p]{\sum_{u\in2^n}\pnp{\f\uhr_{u\appinf}}^p}
=2^{n(p-1)/p}\cdot2^{n/p}\sqrt[p]{2^{-n}\sum_{u\in2^n}\pnp{\f\uhr_{u\appinf}}^p}\\
&\geq2^{n(p-1)/p}\cdot2^{n/p}\cdot2^{-n}\cdot\sum_{u\in2^n}\sqrt[p]{\pnp{\f\uhr_{u\appinf}}^p}
=\sum_{u\in2^n}\pnp{\f\uhr_{u\appinf}}.
\end{align*}
\end{proof}

\begin{lemma}
\label{m_p_is_co_sigma_porous_lemma}
Let $p\in[1, \infty)$ and $k\geq1$ be arbitrary. Then
\begin{equation*}
\bs{\f\in\M_p;\s{\eta\in2^\N;\osc f_n(\eta)\geq k}\text{ is comeager in }(2^\N,\rhoharm)}
\end{equation*}
is co-$\sigma$-porous in $\M_p$.
\end{lemma}

\begin{proof}
First, let us denote
\begin{flalign*}
\hspace{20pt}N_{\f,k}&\coloneqq\s{\eta\in2^\N;\osc f_n(\eta)\geq k}\hspace{10pt}\text{for all }\f\in\M_p,&\\
C&\coloneqq\bs{\f\in\M_p;N_{\f,k}\text{ is comeager in }(2^\N,\rhoharm)},&\\
m_{n,j}&\coloneqq \cl{jp}+n+1\hspace{10pt}\text{for all }n,j\in\N.&
\end{flalign*}
We divide the rest of the proof into six steps.

\begin{enumerate}[leftmargin=*]
\item \textbf{Basic Construction.}
For each $n, j\in\N$, $u\in 2^n$, and $\g\in\M_p$, choose a word $s^{u,j,\g}\in 2^{m_{n,j}}\cap u\appinf$ such that
\begin{equation*}
\g(s^{u,j,\g}) = \max\bs{\g(w); w\in 2^{m_{n,j}}\cap u\appinf}.
\end{equation*}
Set $a\coloneqq 1 - s^{u,j,\g}_{m_{n,j}}$ and $t^{u,j,\g}\coloneqq (s^{u,j,\g}_{<m_{n,j}})\app a$. Since $m_{n,j} \geq n+2$, both $s^{u,j,\g}$ and $t^{u,j,\g}$ belong to $u\appinf$. Define $\e^{u,j,\g}\colon 2^{<\N}\to\R$ by
\begin{equation}
\label{m_p_is_co_sigma_porous:e_def}
\e^{u,j,\g}(v)\coloneqq
\begin{cases}
1, & \text{if } v\in (s^{u,j,\g})\appinf,\\
-1, & \text{if } v\in (t^{u,j,\g})\appinf,\\
0, & \text{otherwise}.
\end{cases}
\end{equation}
Since $s^{u,j,\g}$ and $t^{u,j,\g}$ share the same predecessor of length $m_{n,j}-1$, the function $\e^{u,j,\g}$ satisfies \eqref{eq:prumer} and is a martingale. For every $\g\in\M_p$ and $n,j\in\N$, we define
\begin{equation}
\label{m_p_is_co_sigma_porous:f_def}
\f^{\g,n,j} \coloneqq \g + (1+k)\sum_{u\in2^n}\e^{u,j,\g},
\end{equation}
and set $\mathcal B_{n,j}\coloneqq\bs{B\p{\f^{\g,n,j}, 2^{-n-4-j}};\g\in\M_p}$.

\item \textbf{Jumps are certain and don't displace much.}
We prove the following two properties for all $n,j\in\N$, $\g\in\M_p$ and $\h\in B\p{\f^{\g,n,j}, 2^{-n-4-j}}$.
\begin{flalign*}
\tag{I}\label{m_p_is_co_sigma_porous:I}
&P\bp{\bigcup\bs{u\appinf\cap 2^\N; u\in 2^n, \h(s^{u,j,\g}) \geq \h(u)+k}}\geq 1-2^{-n-1},&\\
\tag{II}\label{m_p_is_co_sigma_porous:II}
&B\p{\f^{\g,n,j}, 2^{-n-4-j}} \subset B\p{\g, (1+k)\cdot 2^{-j+1}}.&
\end{flalign*}

For \eqref{m_p_is_co_sigma_porous:I}, let $\h\in B\p{\f^{\g,n,j}, 2^{-n-4-j}}$. Since $\g(u)$ is the average of $\g$ over $2^{m_{n,j}}\cap u\appinf$, the maximality of $\g(s^{u,j,\g})$ gives $\g(s^{u,j,\g})\geq \g(u)$. Moreover, $\e^{u,j,\g}(s^{u,j,\g})=1$ while $\e^{u,j,\g}(u)=0$. Thus,
\begin{equation}
\label{m_p_is_co_sigma_porous:maximality_corollary}
\f^{\g,n,j}(s^{u,j,\g}) - \f^{\g,n,j}(u) = \g(s^{u,j,\g}) + (1+k) - \g(u) \geq 1+k.
\end{equation}
Setting $\mathbf{d}\coloneqq \h-\f^{\g,n,j}$, whenever $u\in 2^n$ satisfies $\h(s^{u,j,\g}) < \h(u)+k$, we have
\begin{equation*}
1 \leq \bp{\f^{\g,n,j}(s^{u,j,\g}) - \h(s^{u,j,\g})} + \bp{\h(u) - \f^{\g,n,j}(u)} \leq |\mathbf{d}(s^{u,j,\g})| + |\mathbf{d}(u)|.
\end{equation*}
Since $n < m_{n,j}$, we have $|\mathbf{d}(s^{u,j,\g})| \leq 2^{m_{n,j}/p}\pnp{\mathbf{d}\uhr_{u\appinf}}$ and $|\mathbf{d}(u)| \leq 2^{n/p}\pnp{\mathbf{d}\uhr_{u\appinf}} \leq 2^{m_{n,j}/p}\pnp{\mathbf{d}\uhr_{u\appinf}}$. Hence,
\begin{equation*}
1 \leq 2 \cdot 2^{m_{n,j}/p}\pnp{\mathbf{d}\uhr_{u\appinf}}.
\end{equation*}
Let $E\coloneqq \{u\in 2^n; \h(s^{u,j,\g}) < \h(u)+k\}$. Using Lemma \ref{martingales_are_not_sigma_porous_infty_lp_lemma},
\begin{align*}
P\bp{\bigcup_{u\in E} u\appinf\cap 2^\N}
&= 2^{-n}|E| \leq 2^{-n}\sum_{u\in 2^n} 2\cdot 2^{m_{n,j}/p}\pnp{\mathbf{d}\uhr_{u\appinf}} \\
&\leq 2^{-n}\cdot 2\cdot 2^{m_{n,j}/p}\cdot 2^{n(p-1)/p}\pnp{\mathbf{d}} = 2\cdot 2^{(m_{n,j}-n)/p}\pnp{\h-\f^{\g,n,j}}\\
&\leq 2^{1 + \frac{\cl{jp}+1}{p}} \cdot 2^{-n-4-j} \leq 2^{1 + j + \frac{2}{p} - n - 4 - j} = 2^{\frac{2}{p} - 3 - n} \leq 2^{-n-1},
\end{align*}
where the last inequality holds since $p \geq 1 \implies \frac{2}{p} - 3 \leq -1$. Since $2^n\setminus E = \{u\in 2^n; \h(s^{u,j,\g}) \geq \h(u)+k\}$, taking complements gives \eqref{m_p_is_co_sigma_porous:I}.

For \eqref{m_p_is_co_sigma_porous:II}, note that the martingales $(\e^{u,j,\g})_{u\in 2^n}$ have pairwise disjoint supports on levels $\geq m_{n,j}$, with $P((s^{u,j,\g})\appinf) = P((t^{u,j,\g})\appinf) = 2^{-m_{n,j}}$. Therefore,
\begin{align*}
\pnp{\f^{\g,n,j}-\g}
&= (1+k)\Bpnp{\sum_{u\in2^n}\e^{u,j,\g}} = (1+k)\sqrt[p]{2^n\cdot 2\cdot 2^{-m_{n,j}}}\\
&= (1+k)\sqrt[p]{2^{n+1-(\cl{jp}+n+1)}} \leq (1+k)2^{-j}.
\end{align*}
Since $2^{-n-4-j} \leq (1+k)2^{-j}$, the triangle inequality gives
\begin{equation*}
B\p{\f^{\g,n,j}, 2^{-n-4-j}} \subset B\p{\g, (1+k)2^{-j} + 2^{-n-4-j}} \subset B\p{\g, (1+k)2^{-j+1}},
\end{equation*}
which proves \eqref{m_p_is_co_sigma_porous:II}.

\item \textbf{Definition of $A_n$.}
We define
\begin{equation}
\label{m_p_is_co_sigma_porous:A_n_definition}
A_n\coloneqq\bigcup_{j\geq n}\bigcup\mathcal B_{n,j},\quad n\in\N.
\end{equation}

\item \textbf{Co-porosity of $A_n$.}
Fix $n\in\N$, $\g\in\M_p$, and $r\in(0,2^{-n})$. Set
\begin{equation}
\label{m_p_is_co_sigma_porous:j_def}
j\coloneqq\min\s{i\in\N;(1+k)2^{-i+1}\leq r}\qquad\text{and}\qquad\alpha\coloneqq\frac{2^{-n-6}}{1+k}.
\end{equation}
Since $r < 2^{-n}$ and $k\geq 1$, we have $(1+k)2^{-n+1} \geq 4\cdot 2^{-n} > r$, so $j > n$. By the minimality of $j$, $(1+k)2^{-j+2} > r$, so $2^{-j} > \frac{r}{4(1+k)}$.

By \eqref{m_p_is_co_sigma_porous:II}, $B(\f^{\g,n,j}, 2^{-n-4-j}) \subset B(\g, (1+k)2^{-j+1}) \subset B(\g, r)$. Moreover,
\begin{equation*}
\alpha r = \frac{2^{-n-6}}{1+k}r < 2^{-n-4}\cdot 2^{-j} = 2^{-n-4-j}.
\end{equation*}
Since $j \geq n$, we have $B(\f^{\g,n,j}, 2^{-n-4-j}) \subset \bigcup \mathcal B_{n,j} \subset A_n$, and therefore
\begin{equation*}
B(\f^{\g,n,j},\alpha r) \subset B(\f^{\g,n,j}, 2^{-n-4-j}) \subset B(\g, r) \cap A_n = B(\g,r)\setminus (\M_p\setminus A_n).
\end{equation*}
Since $\alpha > 0$ depends only on $n$ and $k$, the set $\M_p\setminus A_n$ is $2\alpha$-porous at each point $\g\in\M_p$, showing that $A_n$ is co-porous.

\item \textbf{Comeagerness of divergence sets.}
We show that for every $\h\in\bigcap_{n\in\N}A_n$, the set $N_{\h,k}$ is comeager in $(2^\N,\rhoharm)$. Fix $\h\in\bigcap_{n\in\N}A_n$. For each $n\in\N$, using \eqref{m_p_is_co_sigma_porous:A_n_definition}, we choose $j_n\geq n$ and $\g_n\in\M_p$ such that $\h\in B(\f^{\g_n,n,j_n}, 2^{-n-4-j_n})$. For brevity, we denote $s^{u,n}\coloneqq s^{u,j_n,\g_n}$ and set
\begin{equation*}
U_n \coloneqq \bs{u\in 2^n; \h(s^{u,n})\geq \h(u)+k}.
\end{equation*}
By \eqref{m_p_is_co_sigma_porous:I}, $P\bp{\bigcup_{u\in U_n} u\appinf\cap 2^\N} \geq 1 - 2^{-n-1}$. We define
\begin{equation*}
G_n \coloneqq \bigcup_{u\in U_n} (s^{u,n})\appinf \cap 2^\N.
\end{equation*}
Each $G_n$ is open in $2^\N$. For any $m\in\N$, the set $W_m \coloneqq \bigcup_{n\geq m} G_n$ is open. To show density of $W_m$, let $v\in 2^l$ be an arbitrary finite word. For any $n \geq \max\{m, l\}$,
\begin{equation*}
P\bp{2^\N \setminus \bigcup_{u\in U_n} u\appinf\cap 2^\N} \leq 2^{-n-1} < 2^{-l} = P\bp{v\appinf\cap 2^\N}.
\end{equation*}
Thus, $v\appinf\cap 2^\N$ cannot be contained in the complement of $\bigcup_{u\in U_n} u\appinf\cap 2^\N$, so there exists $w\in 2^n \cap v\appinf$ belonging to $U_n$. For this $w$, we have $(s^{w,n})\appinf \cap 2^\N \subset G_n \subset W_m$, and since $s^{w,n} \in w\appinf \subset v\appinf$, it follows that $W_m \cap (v\appinf \cap 2^\N) \neq \es$. Thus $W_m$ is open and dense in $2^\N$.

By the Baire Category Theorem, $G \coloneqq \bigcap_{m=1}^\infty W_m = \limsup_{n\to\infty} G_n$ is comeager in $2^\N$. If $\eta\in G$, then $\eta\in G_n$ for infinitely many $n\in\N$. Whenever $\eta\in G_n$, there exists $u\in U_n$ such that $\eta_{\leq m_{n,j_n}} = s^{u,n}$ and $\eta_{\leq n} = u$, which yields
\begin{equation*}
h_{m_{n,j_n}}(\eta) - h_n(\eta) = \h(s^{u,n}) - \h(u) \geq k.
\end{equation*}
Since $m_{n,j_n} > n \to \infty$, it follows that $\osc h_l(\eta) \geq k$. Therefore $G \subset N_{\h,k}$, proving that $N_{\h,k}$ is comeager.

\item \textbf{Conclusion.}
From Step 5, we have $\bigcap_{n\in\N}A_n \subset C$, which means
\begin{equation*}
\M_p\setminus C \subset \M_p\setminus \bigcap_{n\in\N}A_n = \bigcup_{n\in\N} (\M_p\setminus A_n).
\end{equation*}
Since each $\M_p\setminus A_n$ is porous by Step 4, $\M_p\setminus C$ is $\sigma$-porous in $\M_p$. Hence $C$ is co-$\sigma$-porous in $\M_p$.
\end{enumerate}
\end{proof}

\MpIsCoSigmaPorous*

\begin{proof}
For each $k\in\N$, let
\begin{equation*}
C_k \coloneqq \bs{\f\in\M_p;\s{\eta\in2^\N;\osc f_n(\eta)\geq k}\text{ is comeager in }(2^\N,\rhoharm)}.
\end{equation*}
By Lemma \ref{m_p_is_co_sigma_porous_lemma}, each $C_k$ is co-$\sigma$-porous in $\M_p$. If $\f\in\bigcap_{k\in\N} C_k$, then for every $k\in\N$ the set $N_{\f,k} = \s{\eta\in2^\N;\osc f_n(\eta)\geq k}$ is comeager in $(2^\N,\rhoharm)$. Consequently, the countable intersection
\begin{equation*}
\bigcap_{k\in\N} N_{\f,k} = \bs{\eta\in2^\N;\osc f_n(\eta)=\infty}
\end{equation*}
is comeager in $(2^\N,\rhoharm)$. Thus $\bigcap_{k\in\N} C_k \subset \bs{\f\in\M_p;\s{\eta\in2^\N;\osc f_n(\eta)=\infty}\text{ is comeager in }(2^\N,\rhoharm)}$. Since the countable intersection of co-$\sigma$-porous sets is co-$\sigma$-porous, the proof is complete.
\end{proof}

\begin{lemma}
\label{mi_pr_not_sigma_porous_lemma}
Let $\f\in\M_\infty$, $m\in\N$, $u\in2^m$ and $n,k\in\N_0$ be arbitrary. Let us further denote
\begin{equation*}
M\coloneqq\bs{v\in2^{l}\cap u\appinf;l\geq m,\forall j\in[m:l]:f_j(v)\geq \f(u)-2^{-n}}
\end{equation*}
and
\begin{equation*}
s\coloneqq\sup\s{\f(v);v\in M}.
\end{equation*}
Then given any $v\in M$ such that $s-\f(v)\leq2^{-k-n}$, it holds that
\begin{equation*}
\forall w\in2^{\leq k}:s-\f(v\app w)\leq2^{-n},\quad\text{in particular}\quad v\app w\in M.
\end{equation*}
\end{lemma}

\begin{proof}
We prove the lemma by induction on $k$. For $k=0$, the only word is $w=\es$, and since $v\in M$, it evidently holds that $v\app\es = v\in M$ and $s-\f(v)\le 2^{-n}$.

Let us now assume that the statement holds for some $k\in\N_0$. Let $v\in M$ be such that $s-\f(v)\leq 2^{-(k+1)-n}$. For each $a\in\s{0,1}$, we must have $\f(v\app a)\leq s$ (otherwise $\f(v\app a) > s \geq \f(u) \geq \f(u)-2^{-n}$, implying $v\app a \in M$, which contradicts the definition of $s$). 

By the martingale property, $\f(v) = \frac{1}{2}(\f(v\app 0) + \f(v\app 1))$. Thus, for every $a\in\s{0,1}$,
\begin{equation*}
\f(v\app a) = 2\f(v) - \f(v\app(1-a)) \geq 2(s - 2^{-k-1-n}) - s = s - 2^{-k-n},
\end{equation*}
which means $s - \f(v\app a) \leq 2^{-k-n}$. Since $s \ge \f(u)$, we also have $\f(v\app a) \ge s - 2^{-k-n} \ge \f(u) - 2^{-n}$, so $v\app a \in M$.

Applying the induction hypothesis to each $v\app a \in M$ with parameter $k$, we obtain that for all $w'\in 2^{\leq k}$, $s-\f(v\app a\app w')\leq 2^{-n}$ and $v\app a\app w' \in M$. Since every nonempty $w\in 2^{\leq k+1}$ can be written as $a\app w'$ for some $a\in\s{0,1}$ and $w'\in 2^{\leq k}$, and for $w=\es$ we have $s-\f(v)\leq 2^{-(k+1)-n}\leq 2^{-n}$, the induction step is complete.
\end{proof}

\MiPrNotSigmaPorous*

\begin{proof}
For contradiction, let us assume that $A$ is $\sigma$-porous. We use Corollary \ref{zajicek_corollary} to find sets $A_n$ such that $\bigcup_{n\in\N}A_n=A$ and for each $n\in\N$,
\begin{equation}
\label{mi_pr_not_sigma_porous:porosity_assumption}
\forall x\in\PriSp\forall r\in(0,2^{-n})\exists y\in\PriSp:\overline{B\p{y,2^{-n}\cdot r}}\subset B(x,r)\setminus A_n.
\end{equation}
We will construct a point $z\equiv (\g,\nu)\in\PriSp$ such that 
\begin{equation*}
\osc g_j(\nu)=0\hspace{30pt}\text{and}\hspace{30pt}z\notin\bigcup_{n\in\N}A_n,
\end{equation*}
thereby obtaining a contradiction.

We inductively construct sequences $(m_n)_{n=1}^\infty \subset \N$ with $m_n < m_{n+1}$, $(z_n)_{n=1}^\infty \equiv (\g^n,\nu^n) \in \PriSp$, and $(s_n)_{n=1}^\infty \subset \R$ such that for all $n\in\N$:
\begin{flalign*}
\tag{I}\label{mi_pr_not_sigma_porous:I}
&\overline{B\p{z_{n+1},2^{-m_{n+1}}}}\subset B\p{z_n,2^{-m_n}}\setminus A_n,&\\
\tag{II}\label{mi_pr_not_sigma_porous:II}
&s_{n+1}\geq s_n - 2^{-n+2},&\\
\tag{III}\label{mi_pr_not_sigma_porous:III}
&\forall j\in[m_n:m_{n+1}]:\min\s{s_n, s_{n-1}} - 6\cdot 2^{-n} \leq \g^{n+1}_{j}(\nu^{n+1}) \leq s_n + 2^{-n},&\\
&\hspace{292pt}(n\geq2)\hfill
\end{flalign*}

\textbf{Initialization ($n=1$):} Choose $z_1\equiv(\g^1,\nu^1)\in\PriSp$ arbitrarily, set $m_1\coloneqq 2$, $u_1\coloneqq\nu^1_{\leq m_1}$,
\begin{equation*}
M_1 \coloneqq \bs{v\in 2^l\cap u_1\appinf; l\geq m_1, \forall j\in[m_1:l]: g^1_j(v)\geq \g^1(u_1)-2^{-1}},
\end{equation*}
and $s_1 \coloneqq \sup\{\g^1(w); w\in M_1\}$.

\textbf{Induction step:} Assume $m_n$, $z_n\equiv(\g^n,\nu^n)$, $u_n\coloneqq\nu^n_{\leq m_n}$, $M_n$, and $s_n\coloneqq\sup\{\g^n(w); w\in M_n\}$ are given. By definition of the supremum, there exists $v'\in M_n$ such that $s_n - \g^n(v') \leq 2^{-2n}$. If $|v'| < \max\{m_n, n+1\}$, we can repeatedly extend $v'$: at each step, choose a child $v'\app a$ ($a\in\{0,1\}$) such that $\g^n(v'\app a)\ge \g^n(v')$, which exists by $\g^n(v')=\frac{1}{2}(\g^n(v'\app 0)+\g^n(v'\app 1))$. Since $\g^n(v'\app a)\ge \g^n(v')\ge \g^n(u_n)-2^{-n}$, each extension remains in $M_n$. We thus obtain $v\in 2^l \cap M_n$ with $l\geq \max\{m_n, n+1\}$ satisfying $s_n-\g^n(v)\leq 2^{-2n}$.

Fix any $\tilde{v}\in 2^\N \cap v\appinf$. Applying \eqref{mi_pr_not_sigma_porous:porosity_assumption} to $x\coloneqq (\g^n,\tilde{v})$ and $r\coloneqq 2^{-l} < 2^{-n}$, we obtain $z_{n+1}\equiv (\g^{n+1},\nu^{n+1})\in\PriSp$ such that
\begin{equation}
\label{mi_pr_not_sigma_porous:porosity_corollary}
\overline{B\p{z_{n+1},2^{-n}\cdot 2^{-l}}}\subset B\bp{(\g^n,\tilde{v}),2^{-l}}\setminus A_n.
\end{equation}
Set $m_{n+1}\coloneqq l+n > m_n$. Since $v \in u_n\appinf$, we have $B_{\rhoexp}(\tilde{v}, 2^{-l}) \subset B_{\rhoexp}(\nu^n, 2^{-m_n})$, which establishes \eqref{mi_pr_not_sigma_porous:I}.

Next, define $u_{n+1}\coloneqq\nu^{n+1}_{\leq m_{n+1}}$,
\begin{equation*}
M_{n+1}\coloneqq\bs{y\in 2^{l'}\cap u_{n+1}\appinf; l'\ge m_{n+1}, \forall j\in[m_{n+1}:l']: g^{n+1}_j(y)\ge \g^{n+1}(u_{n+1})-2^{-(n+1)}},
\end{equation*}
and $s_{n+1}\coloneqq\sup\{\g^{n+1}(y); y\in M_{n+1}\}$. Note that $\rhoexp(\nu^{n+1}, \tilde{v}) < 2^{-l}$ implies $\nu^{n+1}_{\leq l} = \tilde{v}_{\leq l} = v$, so $u_{n+1} = v \app w$ for some $w\in 2^n$. Applying Lemma \ref{mi_pr_not_sigma_porous_lemma} with $k\coloneqq n$, we get $s_n - \g^n(u_{n+1})\leq 2^{-n}$. Since $\|\g^{n+1}-\g^n\|_\infty \leq 2^{-l} \leq 2^{-n-1}$,
\begin{equation*}
\g^{n+1}(u_{n+1}) \geq \g^n(u_{n+1}) - 2^{-n-1} \geq s_n - 2^{-n} - 2^{-n-1} = s_n - 3\cdot 2^{-n-1} \geq s_n - 2^{-n+1}.
\end{equation*}
Since $u_{n+1}\in M_{n+1}$ trivially, $s_{n+1} \geq \g^{n+1}(u_{n+1}) \geq s_n - 2^{-n+2}$, which proves \eqref{mi_pr_not_sigma_porous:II}.

To verify \eqref{mi_pr_not_sigma_porous:III} for $n\geq 2$, observe that $\|\g^{n+1}-\g^n\|_\infty\leq 2^{-l}\leq 2^{-n}$ and $\g^n(u_n) \geq s_{n-1}-2^{-n+2}$ from \eqref{mi_pr_not_sigma_porous:II} at the previous step. For $j\in[m_n:l]$, the prefix $v_{\leq j}\in M_n$, so
\begin{equation*}
\min\s{s_n, s_{n-1}}-5\cdot 2^{-n} \leq \g^n(u_n)-2^{-n} \leq \g^n_j(\nu^{n+1}) \leq s_n.
\end{equation*}
For $j\in[l:m_{n+1}]$, Lemma \ref{mi_pr_not_sigma_porous_lemma} yields $s_n - 2^{-n} \leq \g^n_j(\nu^{n+1}) \leq s_n$. Combining both intervals with $\|\g^{n+1}-\g^n\|_\infty \leq 2^{-n}$ yields \eqref{mi_pr_not_sigma_porous:III}.

\textbf{Convergence and contradiction:} By completeness of $\PriSp$ and \eqref{mi_pr_not_sigma_porous:I}, there is a point
\begin{equation*}
(\g,\nu)\equiv z\in\bigcap_{n\in\N}\overline{B(z_n,2^{-m_n})}.
\end{equation*}
From \eqref{mi_pr_not_sigma_porous:I}, $\overline{B(z_{n+1},2^{-m_{n+1}})} \subset B(z_n,2^{-m_n})\setminus A_n$, so $z\notin\bigcup_{n\in\N}A_n$.

Since $\|\g^n\|_\infty \leq \|\g^1\|_\infty + \sum_{k=1}^\infty 2^{-m_k} \leq \|\g^1\|_\infty + 1/2$, the sequence $(s_n)$ is bounded from above. By \eqref{mi_pr_not_sigma_porous:II}, $b_n \coloneqq s_n - \sum_{k=n}^\infty 2^{-k+2} = s_n - 2^{-n+3}$ is non-decreasing and bounded from above, hence converges to a finite limit $S\in\R$. Thus $s_n \to S$.

For any $n\geq 2$, $z\in \overline{B(z_{n+2}, 2^{-m_{n+2}})} \subset B(z_{n+1}, 2^{-m_{n+1}})$, so $\rhoexp(\nu, \nu^{n+1}) < 2^{-m_{n+1}}$. This implies $\nu_{\leq m_{n+1}} = \nu^{n+1}_{\leq m_{n+1}}$, and thus $g^{n+1}_j(\nu) = g^{n+1}_j(\nu^{n+1})$ for all $j\leq m_{n+1}$. Moreover, $\|\g - \g^{n+1}\|_\infty \leq 2^{-m_{n+1}} \leq 2^{-n}$. Combining this with \eqref{mi_pr_not_sigma_porous:III}, for all $j\in[m_n:m_{n+1}]$:
\begin{equation*}
\min\s{s_n, s_{n-1}} - 7\cdot 2^{-n} \leq g_j(\nu) \leq s_n + 2\cdot 2^{-n}.
\end{equation*}
Taking $j\to\infty$ forces $n\to\infty$. By the squeeze theorem, $\lim_{j\to\infty} g_j(\nu) = S$, so $\osc g_j(\nu) = 0$. Hence $z\in A$, which contradicts $z\notin \bigcup_{n\in\N}A_n$.
\end{proof}

\begin{lemma}
\label{limit_porosity_criterion_lemma}
Let $\varphi\colon\N\ra\R^+$ be a decreasing function such that $\varphi(n)\ra0$ and
\begin{equation}
\label{limit_porosity_criterion_lemma:contradiction_assumption}
\forall q>0\forall n_0\in\N\exists n\geq n_0:\varphi(n+1)<q\cdot\varphi(n).
\end{equation}
Also, let us consider any $\es\neq A\subset2^\N$. Then $A$ is not porous in $(2^\N,\psi_\varphi)$.
\end{lemma}

\begin{proof}
Let us consider any $\nu\in A$ and $q\in(0,1/4)$. We show that $A$ is not $4q$-porous at $\nu$. To apply Lemma \ref{porosity_condition_lemma}, we set $\alpha\coloneqq 2q$ and $\varepsilon\coloneqq q$. Let $r_0>0$ be arbitrary. By \eqref{limit_porosity_criterion_lemma:contradiction_assumption}, there exists $n\in\N$ such that 
\begin{equation}
\label{limit_porosity_criterion_lemma:contradiction_corollary}
\varphi(n)<r_0\hspace{20pt}\text{and}\hspace{20pt}\varphi(n+1)<q\cdot\varphi(n).
\end{equation}
For $r\coloneqq \varphi(n) < r_0$ and any $\eta\in B_{\psi_\varphi}(\nu,\varphi(n))$, we have $\eta_{\leq n} = \nu_{\leq n}$. Hence $\psi_\varphi(\eta,\nu) \leq \varphi(n+1) < q\cdot\varphi(n)$, which implies
\begin{equation*}
\nu \in B_{\psi_\varphi}\bp{\eta, q\cdot\varphi(n)} = B_{\psi_\varphi}\bp{\eta, (\alpha-\varepsilon)r}.
\end{equation*}
Since $\nu\in A$, no ball $B_{\psi_\varphi}(\eta, (\alpha-\varepsilon)r) \subset B_{\psi_\varphi}(\nu, r)$ can be disjoint from $A$. By Lemma \ref{porosity_condition_lemma}, $A$ is not $4q$-porous at $\nu$. Since $q>0$ was arbitrarily small, $p(A,\nu)=0$, so $A$ cannot be porous on $2^\N$.
\end{proof}

\LimitPorosityCriterionTheorem*

\begin{proof}
Let $A\subset2^\N$ be $2c$-porous in $(2^\N,\psi_1)$ for some $c>0$. If $A=\es$, it is trivially porous in $(2^\N,\psi_2)$. In the sequel, we assume $A\neq\es$.

By Lemma \ref{limit_porosity_criterion_lemma}, since $A$ is porous in $(2^\N,\psi_1)$, $\varphi_1$ cannot satisfy \eqref{limit_porosity_criterion_lemma:contradiction_assumption}. Combining this with \eqref{limit_porosity_criterion_theorem:varphi12relation}, there exists $q\in(0,1)$ such that
\begin{equation}
\label{limit_porosity_criterion_theorem:exp_comparison}
\forall i\in\s{1,2}\forall n\in\N:\varphi_i(n+1)\geq q\varphi_i(n).
\end{equation}
We show that $A$ is $2q^2c$-porous in $(2^\N,\psi_2)$ using Lemma \ref{porosity_condition_lemma}. Let $\nu\in2^\N$ and $\varepsilon\in(0,q^2c)$ be arbitrary. Since $A$ is $2c$-porous in $(2^\N,\psi_1)$, there exists $r_1>0$ such that
\begin{equation}
\label{limit_porosity_criterion_theorem:psi1_porosity_corollary}
\forall r\in(0,r_1)\exists\eta\in2^\N:B_{\psi_1}\bp{\eta,r\p{c-\varepsilon}}\subset B_{\psi_1}\p{\nu,r}\setminus A.
\end{equation}
Let $n_1\in\N$ be such that $\varphi_1(n_1)\leq \min\s{r_1, \varphi_1(n_0)}(c-\varepsilon)$, and set $r_0\coloneqq \min\s{\varphi_2(n_1),\varphi_2(n_0)}$.

Let $r\in(0,r_0)$ be arbitrary. We define
\begin{equation*}
    n_2\coloneqq\max\s{n\in\N;q^2(c-\varepsilon)r\leq \varphi_2(n)}.
\end{equation*}
By \eqref{limit_porosity_criterion_theorem:exp_comparison}, $\varphi_2(n_2) \leq \frac{1}{q}\varphi_2(n_2+1) \leq q(c-\varepsilon)r$. Since $r < \varphi_2(n_1)$, we have $n_2 \ge n_1$. Next, set
\begin{equation*}
    n_3\coloneqq\max\Bs{n\in\N;\frac{\varphi_1(n_2)}{c-\varepsilon}\leq\varphi_1(n)}.
\end{equation*}
Since $\frac{\varphi_1(n_2)}{c-\varepsilon} \leq \frac{\varphi_1(n_1)}{c-\varepsilon} \leq \varphi_1(n_0)$, we have $n_0 \leq n_3 \leq n_2$. By \eqref{limit_porosity_criterion_theorem:exp_comparison}, $\varphi_1(n_3) \leq \frac{\varphi_1(n_2)}{q(c-\varepsilon)}$. Multiplying \eqref{limit_porosity_criterion_theorem:varphi12relation} from $n_3$ to $n_2-1$ yields $\frac{\varphi_2(n_3)}{\varphi_2(n_2)} \leq \frac{\varphi_1(n_3)}{\varphi_1(n_2)}$, which gives
\begin{equation*}
    \varphi_2(n_3) \leq \varphi_2(n_2)\frac{\varphi_1(n_3)}{\varphi_1(n_2)} \leq \frac{\varphi_2(n_2)}{q(c-\varepsilon)} \leq r.
\end{equation*}
Applying \eqref{limit_porosity_criterion_theorem:psi1_porosity_corollary} for $r'\coloneqq \frac{\varphi_1(n_2)}{c-\varepsilon} \leq r_1$, there exists $\eta\in 2^\N$ such that
\begin{align*}
B_{\psi_2}\p{\eta,(q^2c-\varepsilon) r}
&\subset B_{\psi_2}\p{\eta,q^2(c-\varepsilon) r}
\subset B_{\psi_2}\p{\eta,\varphi_2(n_2)}
= B_{\psi_1}\p{\eta,\varphi_1(n_2)}\\
&\overeq{\subset}{\eqref{limit_porosity_criterion_theorem:psi1_porosity_corollary}} B_{\psi_1}\p{\nu,\frac{\varphi_1(n_2)}{c-\varepsilon}}\setminus A
\subset B_{\psi_1}\p{\nu,\varphi_1(n_3)}\setminus A\\
&= B_{\psi_2}\p{\nu,\varphi_2(n_3)}\setminus A
\subset B_{\psi_2}\p{\nu,r}\setminus A.
\end{align*}
By Lemma \ref{porosity_condition_lemma}, $A$ is $2q^2c$-porous in $(2^\N,\psi_2)$.
\end{proof}

\section{Acknowledgements}

Our thanks belong to Prof. Jiří Spurný for suggesting the topic, for reading some of the proofs and for his countless remarks.

This research did not receive any specific grant from funding agencies in the public, commercial, or not-for-profit sectors.

\bibliographystyle{acm}
\bibliography{martingales}

\end{document}